\documentclass{article}

\usepackage{amsmath, amsfonts, amssymb, amsthm}
\usepackage[inline]{enumitem}
\usepackage{tikz-cd}
\usepackage[left=3.7cm, right=3.7cm, top=3.7cm, bottom=4.2cm]{geometry}
\usepackage[french, english]{babel}

\author{Robin Mader, Terry Gannon, and Arturo Pianzola}

\date{}

\newtheorem{thm}{Theorem}[section]
\newtheorem{lem}[thm]{Lemma}
\newtheorem{prop}[thm]{Proposition}
\newtheorem{cor}[thm]{Corollary}

\theoremstyle{definition}
\newtheorem{de}[thm]{Definition}
\newtheorem{ex}[thm]{Example}
\newtheorem{rem}[thm]{Remark}	
\newtheorem*{notat}{Notation}

\def\id{\ensuremath{\operatorname{id}}}
\def\End{\ensuremath{\operatorname{End}}}
\def\Hom{\ensuremath{\operatorname{Hom}}}
\def\Out{\ensuremath{\operatorname{Out}}}
\def\Aut{\ensuremath{\operatorname{Aut}}}
\def\Spec{\ensuremath{\operatorname{Spec}}}

\def\BAut{\ensuremath{\operatorname{\mathbf{Aut}}}}
\def\BHom{\ensuremath{\operatorname{\mathbf{Hom}}}}

\def\Rva{\ensuremath{R\text{-\textup{\textsf{va}}}}}

\def\Rcva{\ensuremath{R\text{-\textup{\textsf{cva}}}}}

\def\Scva{\ensuremath{S\text{-\textup{\textsf{cva}}}}}
\def\kcva{\ensuremath{k\text{-\textup{\textsf{cva}}}}}
\def\kalg{\ensuremath{k\text{-\textup{\textsf{alg}}}}}
\def\Rfva{\ensuremath{R\text{-\textup{\textsf{va}}}}^F}
\def\Sfva{\ensuremath{S\text{-\textup{\textsf{va}}}}^F}
\def\grp{\ensuremath{\text{\textup{\textsf{grp}}}}}
\def\set{\ensuremath{\text{\textup{\textsf{set}}}}}
\def\RLie{\ensuremath{R\text{-\textup{\textsf{Lie}}}}}
\def\Ralg{\ensuremath{R\text{-\textup{\textsf{alg}}}}}
\def\BF{\ensuremath{\mathbf{F}}}
\def\BG{\ensuremath{\mathbf{G}}}
\def\BH{\ensuremath{\mathbf{H}}}

\def\BGL{\ensuremath{\mathbf{G}\mathbf{L}}}

\def\GVac{\ensuremath{{\mathbf{G}^{\textup{vac}}}}}

\def\fg{\ensuremath{\mathfrak{g}}}

\def\Z{\ensuremath{\mathbb Z}}
\def\hZ{\ensuremath{\widehat{\mathbb Z}}}
\def\hS{\ensuremath{\widehat S}}
\def\tH{\ensuremath{\widetilde {H}}}

\def\cO{\ensuremath{\mathcal{O}}}
\def\uD{\ensuremath{\underline{D}}}

\title{Descent Theory for Vertex Algebras}

\begin{document}

\maketitle

\begin{abstract}
\noindent Vertex algebras can be defined over any differential commutative ring. We develop the general descent theory for vertex algebras over such bases. We apply this to the classification of twisted forms of affine and Heisenberg vertex algebras, and to reinterpret and generalize a correspondence of Li.
\end{abstract}

\begin{otherlanguage}{french}
\begin{abstract}
\noindent Les alg\`ebres vertex peuvent \^etre d\'efinies sur tout anneau commutatif diff\'erentiel. On d\'eveloppe la th\'eorie g\'en\'erale de la descente pour les alg\`ebres vertex sur de telles bases. On applique celle-ci ensuite \`a la classification des formes tordues des alg\`ebres vertex affines et de Heisenberg, et \`a la r\'einterpr\'etation et \`a la g\'en\'eralisation d'une correspondance de Li.
\end{abstract}
\end{otherlanguage}

\section{Introduction} Defining vertex operator algebras (VOAs) over scalars other than $\Bbb{C}$ goes back to the beginning of the theory. The monstrous moonshine conjectures, relating representations of the monster finite group ${\Bbb M}$ to modular functions, were the primary motivation for introducing VOAs, and were largely explained by the existence of a certain VOA $V^\natural$ with automorphism group ${\Bbb M}$.
The original construction of $V^\natural$ works over $\mathbb Q$, see Chapter 12 of \cite{FLM}.
Somewhat later, Ryba \cite{Ry} proposed analogous conjectures in positive characteristic, and it was quickly realized (see e.g.\ \cite{BR}) that these suggested the existence of an integral form of $V^\natural$. Borcherds and Ryba \cite{BR} proved that the construction in \cite{FLM} actually works over $\mathbb Z[\frac{1}{2}]$. The integral form was finally constructed in \cite{Car}, concluding the proof of Ryba's conjectures.

An integral form for the lattice vertex algebras (i.e.\ the lattice VOAs but ignoring the conformal vector $\omega$) is constructed in the paper  \cite{Bor} that introduced vertex algebras. Since then integral forms for many familiar examples have been studied, and through this so have their versions in positive characteristic.  However, the nature of base change has not been exploited. As we shall see, the natural base for vertex algebras are {\it differential} rings (or schemes). This paper develops that general theory.

Consider for example a vertex algebra $V$ over $\Bbb{C}$. If $R$ is a (commutative unital) ring extension of $\Bbb{C}$, the $R-$module $V_{R,0} =V\otimes_\Bbb{C} R$ has a natural $R-$vertex algebra structure (in the sense of \cite{Mason}) with the $n-$products defined by linear extension of those of $V$:
\begin{equation}(v\otimes r)_n (w\otimes s)=v_nw\otimes rs\label{basic}\end{equation}
for all $v,w\in V$, $r,s\in R$ and $n\in \Z$. This is the obvious example of base change. In particular we encounter in the literature the ``affinization'' of $V$, using $R=\Bbb{C}[t^{\pm1}]$.

However other base changes have appeared. Again take $R=\Bbb{C}[t^{\pm1}]$. Then $R$ together with its derivation $\delta=\frac{d}{dt}$ has a canonical $\Bbb{C}-$vertex algebra structure \cite{Bor} given by $r_ns=\frac{1}{(-n-1)!}\delta^{-n-1}(r)s$ for $n<0$ and 0 otherwise. Therefore so does $V_{R,\delta}=V\otimes_\Bbb{C} R$, with $n-$products defined by
\begin{equation}(v\otimes r)_n(w\otimes s)=\sum_{j\ge0}\frac{1}{j!}v_{n+j}w\otimes\delta^j(r)s\label{diffbasic}\end{equation}
Note that  the $n-$products in \eqref{diffbasic}, unlike those in \eqref{basic}, are $R-$linear only on the second variable. Thus $V_{R,\delta}$ does not have an $R-$structure in the usual sense. Correcting this situation forces one to use as a base  a differential ring, i.e.\ to involve the derivation in the definition of an $R-$vertex algebra. 

This ``differential approach'' was used in \cite{KLP} to explain the counterintuitive \emph{infinite} family of twisted loop algebras attached to the complex $N=4$ conformal superalgebra in \cite{SS}, an object over $\mathbb C$ with connected semisimple automorphism group. One would thus expect the absence of nontrivial twisted loop algebras. We shall see that, and explain why,  this apparently pathological behaviour is quite common in the theory of vertex algebras.

It is also natural to expect that vertex algebras should be objects over commutative vertex algebras. This again takes us into the realm of differential algebras.  Indeed, in characteristic 0 a commutative vertex algebra is the same as a differential ring. In positive characteristic, the derivation should be replaced by iterative Hasse-Schmidt derivations, as explained by Mason \cite{Mason}, Borcherds \cite{Bor}, and discussed below.

Although our approach works for arbitrary differential rings $R$, we specialize in Section \ref{app} to the most important case $R = (\mathbb C[t^{\pm 1}], \frac{d}{dt})$.  We determine the twisted forms  of affine vertex algebras and the isotrivial twisted forms of Heisenberg vertex algebras. Then we show that previous work \cite{Li1} of Li falls out naturally from Galois cohomology considerations.

In a forthcoming paper, we develop some applications of  the theory presented herein to vertex (operator) algebras. For example, we will  use Galois cohomology to find that there are precisely three real forms of the Moonshine VOA $V^\natural$: The one coming from \cite{FLM} has automorphism group the monster $\mathbb{M}$,  and the others have automorphism group an extension of the Baby Monster and the first Conway group respectively.  
The presentation of vertex algebras by means of descent on a given $V$ yields descent data on the corresponding Zhu algebra. This inevitably leads to the appearance of Azumaya algebras and their representation theory. The simplest example is the (unique) non-split  real form of the simple VOA of $sl(2)$ at level 1. Its Zhu algebra is $\mathbb{R} \oplus \Bbb{H}$, where $\mathbb H$ is the real quaternion algebra.
 
\bigskip
\section{Vertex rings}

\subsection{Vertex rings and  the canonical Hasse-Schmidt derivation}
The language of vertex rings \cite{Mason} lends itself perfectly to our formalism.
Below we recapitulate their definition and some of their properties that will be used throughout.

\begin{de}
	A \emph{vertex ring} $(V,Y,\mathbf 1)$ consists of an abelian group $V$, an additive map $Y(\cdot, z)$ that assigns to every $v \in V$ a formal power series $Y(v,z) = \sum_{n \in \mathbb Z} v_n z^{-n-1}$ whose coefficients are group homomorphisms $v_n \colon V \to V$, and a distinguished element $\mathbf 1 \in V$, called the \emph{vacuum vector}, subject to the usual axioms from vertex algebra theory.
	That is, for any $u, v \in V$, we require 	
	\textit{regularity}: $u_n v = 0$ for $n \gg 0$,
	\textit{creation}: $Y(v, z) \mathbf 1 \in v + z V[[z]]$, and
	the \textit{Jacobi identity}, see e.g.\@ \cite[Sec.\ 3.1]{LL}.
A \emph{homomorphism} $f \colon V \to W$ of vertex rings is a group map satisfying $f(u_n v) = f(u)_n f(v)$ for all $u,v \in V$ and $n \in \mathbb Z$, and such that $f(\mathbf 1) = \mathbf 1$.
\end{de}

\cite[Sec.\@ 3]{Mason} The abelian group maps $D_m \colon V \to V,$ $D_m(v) = v_{-m-1} \mathbf 1$, $m \in \mathbb Z_{\geq 0}$, satisfy
    \begin{equation} \label{HasseSchmidtVA} D_0 = \id_V, \quad D_m(u_n v) = \sum_{i+j = m} D_i(u)_n D_j(v), \quad \text{and} \quad D_i D_j = \binom{i+j}{i} D_{i+j}, \end{equation} 
    where $i, j \in \mathbb Z_{\geq 0}$ and $n \in \mathbb Z$. 
    We call the collection $\underline D =(D_i)_{i \geq 0}$ the \emph{canonical Hasse-Schmidt derivation} of $V$.
    In particular, $D = D_1$ is a derivation for all $n-$th products, and it uniquely determines $D_m = \frac{1}{m!} D^m$ if $V$ is a $\mathbb Q-$vector space.
    
\subsection{Vertex algebras over commutative vertex rings}\label{vcvr}
      
Recall that the \textit{centre} $C(V)$ of a vertex ring $V$ is the set of all $v \in V$ such that $Y(v,z_1) Y(u,z_2) = Y(u,z_2) Y(v,z_1)$, or equivalently such that $v_n u = 0 = u_n v$, for all $u \in V$ and $n \geq 0$.
 $V$ is called \textit{commutative} if $V = C(V).$  
            
\begin{thm}[{Cf.\@ \cite[Sec.\@ 5.2]{Mason} and \cite[Sec.\@ 4]{Bor}}] \label{masonord} 
	A commutative vertex ring $R$ is a commutative unital ring with respect to the product $rs = r_{-1} s$, for $r,s \in R$, with $1_R = \mathbf 1.$ Furthermore, the ring $R$ admits an iterative Hasse-Schmidt derivation on this product (Equation (\ref{HasseSchmidtVA}) for $n = -1$).

    Conversely, any commutative unital ring $R$ together with an iterative Hasse-Schmidt derivation $\underline D = (D_i)_{i \geq 0}$ becomes a commutative vertex ring via $\mathbf 1 = 1_R$ and
    \[Y(r,z) s = \sum_{n \geq 0} D_n(r) s z^n, \quad r,s \in R.\]
\end{thm}

 A commutative vertex ring $(R,\underline{D})$ is thus the same as a  differential ring, as long as by derivation one means an iterative Hasse-Schmidt derivation.  

\medskip

We now define the central objects to be studied. Let $R$ be a commutative vertex ring.
   An \emph{$R-$vertex algebra} is a vertex ring $V$ together with a homomorphism $\iota \colon R \to V$ such that $\iota(R) \subset C(V)$. We denote the category of $R-$vertex algebras by $\Rva$, and that of commutative $R-$vertex algebras by $\Rcva.$

 Let us emphasize that this definition of $R-$vertex algebra generalizes the classical definition, as found in e.g.\@ \cite[Sec.\@ 6.2]{Mason} or Borcherds' original work \cite{Bor}, in that the base ring now comes equipped with an iterative Hasse-Schmidt derivation $\underline D$.
For example, the $n-$th product $u_n v$ classically is $R-$bilinear, but for us it need not be $R-$linear in $u$.
Our definition agrees with the usual one when $\underline D = (\id_R, 0, 0, \dots)$ is trivial.
This is clarified by the following characterization of our $R-$vertex algebras.
 
\begin{lem}\ \label{charRva} 
    \begin{enumerate}[label = \textup{\alph*)}]
        \item \label{charRvaa} Any $R-$vertex algebra $(V,\iota)$ carries the structure of an $R-$module via $rv = \iota(r)_{-1} v$, $r \in R$, $v \in V$.
    We have
    \begin{equation} \label{eqchRva} (ru)_n v = \sum_{i \geq 0} D_i(r) u_{n+i} v, \quad \text{and} \quad u_n(rv) = r u_n v, \end{equation}
    for all $r \in R$, and $u,v \in V$.

        \item \label{charRvab} Conversely, any $R-$module structure on a vertex ring $V$ satisfying \eqref{eqchRva} makes $V$ into an $R-$vertex algebra via $\iota(r) = r\mathbf 1$.
        \item \label{charRvac} A map $f \colon (V, \iota) \to (W, \eta)$ is an $R-$vertex algebra homomorphism precisely if it is a vertex ring map and $R-$linear. \qed
    \end{enumerate}
\end{lem}

\begin{notat}
In what follows we usually suppress the structure map $\iota \colon R \to V$ of an $R-$vertex algebra $V$ and write $r_n$ instead of $\iota(r)_n$ for $r \in R$.
Note that $V$ is a vertex ring module over $R$ and that this notation aligns with the usual one for modules.
	Lemma \ref{charRva} also allows us to denote the left multiplication $r_{-1}$ simply by $r$; the creation axiom always ensures $\iota(r) = r \mathbf 1$, so our notation is unambiguous. Finally, there is no harm in denoting the canonical Hasse-Schmidt derivations on both $R$ and $V$ by $\underline D = (D_i)_{i\geq 0}$, because they are preserved by the structure map.
\end{notat}

\subsection{Filtered vertex algebras}

Vertex algebras $V$ over $R=\mathbb C$ often come equipped with a grading $V = \bigoplus_{n \in \mathbb Z}$ satisfying $(V_m)_\ell (V_n) \subset V_{m+n-\ell-1}$.
Lemma \ref{charRva} shows that the notion of grading is incompatible with general bases $R$ if we require $V_n$ to be $R-$submodules; a more suitable notion for $R-$vertex algebras is that of filtration.

Let $V$ be an $R-$vertex algebra.
	A \textit{filtration} on $V$ is a chain of $R-$submodules
	\[ \cdots \subset F_{i-1} V \subset F_i V \subset F_{i+1} V \subset \cdots \]
	such that $\bigcup_{i \in \mathbb Z} F_i V  = V$, $\mathbf 1 \in F_0 V$, and for any $m,n,\ell \in \mathbb Z$: 
    \[(F_m V)_\ell (F_n V) \subset F_{n+m-\ell-1} V. \]
An $R-$vertex algebra $V$ together with a filtration is called \textit{filtered}.
	A homomorphism $f \colon V \to W$ of filtered $R-$vertex algebras is required to satisfy $f(F_i V) \subset F_i W$ for all $i \in \mathbb Z$.
	We denote the category of filtered $R-$vertex algebras by $\Rfva$. 

	Note that any $R-$vertex algebra admits the trivial filtration $F_i V = V$, $i \in \mathbb Z$.
	Thus no generality is lost when we proceed to work with filtered vertex algebras. In what follows we will at times write  $F_i(V)$ instead of $F_iV$ to avoid confusion.
	     
\begin{lem}\label{Lie} 
    Let $V$ be a filtered $R-$vertex algebra.
	Assume that $F_0 V$ is an $R-$vertex subalgebra of $V$.  
    If  the quotient $\mathfrak L(V) = F_1 V / F_0 V$ has no $2$-torsion, it is an $R-$Lie algebra with Lie bracket induced by the $0$-th product. \qed \end{lem}

\begin{lem} \label{comm}
    Let $V$ be a filtered $R-$vertex algebra with $R \simeq F_0V$ and $F_{-1} V = 0$.
    Then the $1$-st product is a symmetric $R-$bilinear map 
			\[ \sigma_V = (\cdot, \cdot) \colon F_1 V \times F_1 V \to F_0 V \simeq R, \] 
			which satisfies the associativity property $ (u, v_0 w) = (u_0 v, w),$
    for all $u,v,w \in F_1 V$; in particular, it induces a symmetric invariant $R-$bilinear form on $\mathfrak L(V) = F_1 V / F_0 V$. \qed
\end{lem}

\section{Automorphisms of vertex algebras}

\subsection{Extension of scalars}\label{ES}

    Let $R$ be a commutative vertex ring. 
     We will often refer to an object $S$ of $\Rcva$ as an \emph{extension} of $R$ and denote it by $S/R.$ 
     
     Let $V$ be an $R-$vertex algebra. The map $Y(\cdot, z) \colon V \otimes_R S \to \End_R (V \otimes_R S) [[z^{\pm 1}]]$ defined by 
     \[ (u \otimes s)_n (v \otimes t) = \sum_{j \geq 0} u_{n+j} v \otimes D^S_j( s)t, \]
     for $u,v \in V$, $s,t \in S$, and $n \in \mathbb Z$, makes $V \otimes_R S$ into an $S-$vertex algebra via $s \mapsto \mathbf 1 \otimes s,$  said to be {\it obtained from $V$ by the base change $S/R.$}    
The canonical Hasse-Schmidt derivation of $V \otimes_R S$ is given by
    
     \[ D^{V \otimes_R S}_n(u \otimes s)  = \sum_{\substack{i+j = n \\ i, j \geq 0}} D^V_i(u) \otimes D^S_j(s)\]
     for $u,s$ and $n$ as before.
     
Finally, if $V$ has a filtration $\{F_i V\}_{i \in \mathbb Z}$, then the images of $F_i(V) \otimes_R S$ in $V \otimes_R S$ define a filtration $\{F_i (V \otimes_R S)\}_{i \in \mathbb Z}$ of $V \otimes_R S.$
For the Lie algebra from Lemma \ref{Lie} we have that, if $S/R$ is flat, then the canonical map $\mathfrak L(V) \otimes_R S \to \mathfrak L(V \otimes_R S)$ is an $S-$Lie algebra isomorphism.
For the filtered $R-$vertex algebra $V$, we define a group functor $\BAut(V) \colon \Rcva \to \grp$ by assigning to any $S$ in $\Rcva$ the group $\BAut(V)(S) = \Aut_{\Sfva}(V \otimes_R S).$
 The automorphism group plays a key role in the descent theory related to twisted forms of $V$ (see Section 4).

\subsection{Vertex group schemes}

\subsubsection{Arc algebras}

 {\rm Let $R = (R, \underline D)$ be a commutative vertex ring.
In order to identify the automorphism group functor $\BAut(V)$ of an $R-$vertex algebra $V$, subject to appropriate finiteness conditions, as a ``closed subgroup'' of the general linear group $\BGL_n$, we must first explain how the latter is represented in the category of commutative $R-$vertex algebras.
In general, suppose we have a representable group functor $\BG$ on the category $\Ralg$ of unital commutative associative $R-$algebras, i.e., an affine $R-$group scheme.
This means we find some $A \in \Ralg$ and an isomorphism of functors
\[ \BG \simeq \Hom_{\Ralg}(A,-) \colon \Ralg \to \grp. \]
We may precompose this functor with the forgetful functor $\Rcva \to \Ralg$ (cf.\@ Theorem \ref{masonord}) and ask if the result is again representable.
Over a field $R = k$ the answer is given by associating to $A$ its so-called \emph{arc algebra} $A^\infty$, see e.g.\@ \cite[Rem.\@ 3.1]{EM}, \cite[Sec.\@ 3.3]{B},  and also Remark \ref{remarc} below. We now show that this procedure carries over to a general commutative vertex ring $R$.

\begin{lem}\label{jetlem}
    For any ordinary $R-$algebra $A$, there exists a  commutative $R-$vertex algebra $A^\infty$ together with an $R-$algebra map $\iota \colon A \to A^\infty$, such that any $R-$algebra homomorphism $f \colon A \to S$ to a commutative $R-$vertex algebra $S$ induces a unique homomorphism $\varphi \colon A^\infty \to S$ of $R-$vertex algebras with $\varphi \iota = f$.
    That is, the forgetful functor $\Rcva \to \Ralg$ admits a left adjoint $-^\infty \colon \Ralg \to \Rcva.$ In particular, any affine $R-$group scheme $\mathbf G$ gives rise to a representable group functor $\mathbf G^\infty \colon \Rcva \to \grp$.

    Moreover, if $A$ is finitely generated as an $R-$algebra, then $A^\infty$ is finitely generated as an $R-$vertex algebra.
\end{lem}

\begin{proof}
We give a description of $A^\infty$ along the lines of the standard literature (see \cite{Voj}, for example). 
	Consider first the case of the polynomial ring  $C = \Z[x_\lambda]$ where $\lambda$ belongs to some index set  $\Lambda.$ We have a commutative vertex ring $(C^\infty, \uD^{C^\infty}),$ where $C^\infty = \Z[x_\lambda^{(j)}]_{_{\lambda \in \Lambda, j \in \Bbb{N}}}$, $\iota$ is given by $x_\lambda \mapsto x_\lambda^{(0)}$,  and  $\underline{D}^{C^\infty}$ is defined by $D^{C^\infty}_i(x_\lambda^{(j)}) =  \binom{i+j}{i}  x_\lambda^{(j +i)}$ and extension to monomials by the Leibniz rule. 

Recall that $R$ carries the canonical Hasse-Schmidt derivation $\uD$. Consider the $R-$algebra $B = R \otimes_\Z C = R[x_\lambda].$ Then $(B^\infty, \uD^{B^\infty})$ is obtained by base change from $C^\infty$, namely 
	$B^\infty = R \otimes_\Z C^\infty = R[x_\lambda^{(j)}]_{_{\lambda \in \Lambda, j \in \Bbb{N}}}$ with iterative Hasse-Schmidt derivation $\underline{D}^{B^\infty} = \underline{D} \otimes \underline{D}^{C^\infty},$ that is ${D}^{B^\infty}_i (ra) =  \sum_{\substack{ \\ p+q = i}}D_p(r)D^{C^\infty}_q(a)$ under the identification $R \otimes_\Z C^\infty = R[x_\lambda^{(j)}]_{_{\lambda \in \Lambda, j \in \Bbb{N}}}$, and where we extend the $\iota$ of $C^\infty$ $R-$linearly. 

Finally assume that $A = B/I$ where $I = \langle f_\mu \rangle_{\mu \in M}.$ Let $I^\infty$ be the ideal of $B^\infty$ generated by all ${D}^{B^\infty}_i(f_\mu).$ Then $I^\infty$ is a vertex algebra ideal of $B^\infty$ and $A^\infty = B^\infty/I^\infty$ with $\iota (rx_\lambda + I) =  rx^{(0)}_\lambda + I^\infty$ is as required.     \end{proof}

 \begin{rem} \label{remarc} Recall the canonical $R-$algebra homomorphism $\iota \colon A \to A^\infty.$ If $\underline{D}$ {\it is trivial,} then $\iota$ is injective and $A^\infty,$ {\it without its vertex ring structure},  is isomorphic to the arc algebra encountered in the literature, where we also find it denoted by $A_\infty$, $J_\infty(A)$, or ${\rm HS}^\infty_{A/R}$ (notably in Vojta's transparent exposition \cite{Voj} on the Hasse-Schmidt analogue of $\Omega_{A/R}$).
	The reader ought to keep in mind that in general $\iota$  {\it need not be injective}. Consider the commutative vertex ring $R = \Z[t]$ with $D = \frac{d}{dt}.$ The homomorphism $\Z[t] \to \Z$  given by $t \mapsto 0$ makes $\Z$ into an object of $\Ralg.$ We have $I = \langle t \rangle.$ Since $1 \in I^\infty$ we get $\Z^\infty = 0.$ 
\end{rem}

\subsubsection{Vertex  schemes and groups}\label{avs}

Let $X$ be a topological space. A \emph{sheaf of (commutative) vertex rings} $\mathcal F$ on $X$ is a sheaf on $X$ whose sections $\mathcal F(U)$ are (commutative) vertex rings for every open set $U \subset X$, and whose restriction maps are homomorphisms of vertex rings. 
This forms a category with the obvious notion of morphism.

A \emph{locally commutative vertex ringed space}, or simply \emph{locally cv ringed space}, is a locally ringed space $(X, \cO_X)$, where the structure sheaf is a sheaf of commutative vertex rings.
A morphism of locally cv ringed spaces $(f, f^\#) \colon (X, \mathcal O_X) \to (Y, \mathcal O_Y)$ is a morphism of locally ringed spaces such that $f^\# \colon \mathcal O_Y \to f_\ast \mathcal O_X$ is a morphism of $Y-$sheaves of vertex rings. 

 Let $(R, \underline{D})$ be a commutative vertex ring. The locally cv ringed space $(\Spec (R), \cO_{\uD})$ 
 is defined as follows: For $f \in R$ there is a unique iterative Hasse-Schmidt derivation ${\uD}^f$ on $R_f$ extending $\uD.$ Then $\cO_{\uD}$ is the unique sheaf of commutative vertex rings on $\Spec(R)$ whose sections over the basic open set $D(f)$ is $(R_f, \underline{D}^f).$ A locally cv ringed space $(X, \cO_X)$ is called an {\it affine vertex scheme} if it is isomorphic to $(\Spec (R), \cO_{\uD})$ for some commutative vertex ring $(R, \uD).$ 
 
 A locally cv ringed space $(X, \cO_X)$ is a {\it vertex scheme} if every point of $X$ admits an open neighbourhood isomorphic to an affine vertex scheme. If $(X, \cO_X)$ is a vertex scheme, the natural map $\Hom\!\big((X, \cO_X), \Spec(R,\uD)\big) \to \Hom_{\mathbb Z\textsf{-cva}}\!\big(R, \cO(X)\big)$ is bijective. In particular, the category of affine vertex schemes is anti-equivalent to that of commutative vertex rings.
 \begin{ex}\label{P1} Let $k$ be a field. For $i = 1,2$ let $R_i = k[x_i]$ and $S_i = k [x_i^{\pm 1}].$ Let $X_i = \Spec(R_i),$ $U_{12} = \Spec(S_1),$ and $U_{21} = \Spec(S_2).$ Consider the $k-$scheme isomorphism $\phi_{12} \colon U_{12} \to U_{21}$ arising  from the $k-$algebra isomorphism $f \colon S_2 \to S_1$ given by $x_2 \mapsto x_1^{-1}.$  Let $D^2 = \frac{d}{d x_2}$ and define $D^1 := f \circ D^2 \circ f^{-1} = - x_1^2 \frac{d}{d x_1}$. 
	 These extend to iterative Hasse-Schmidt derivations $\uD^1$ and $\uD^2$ that preserve $R_1$ and $R_2$ respectively, and such that $f \colon (S_2, \uD^2) \to (S_1, \uD^1)$ is an isomorphism in $\kcva.$ Then $\phi_{12}$ and $\phi_{21} := \phi_{12}^{-1}$  allow us to glue $(X_1, \uD^1)$ and $(X_2, \uD^2)$ along $(U_{12} , \uD^1)$ and $(U_{21}, \uD^2)$  This yields a vertex scheme $(\Bbb{P}_k^1, \uD)$ having the projective line over $k$ as underlying scheme.\end{ex}
 
 Let $(X, \cO_X)$ be a vertex scheme. One defines in the usual fashion the category $X_{\rm vs}$ of vertex schemes over  $(X, \cO_X)$: The objects are morphisms $(Y, \cO_Y) \to (X, \cO_X)$ and the morphisms between objects are the obvious ones. As usual, for $(\Spec(R), \cO_{\uD})$ we will write $R_{\rm vs}$ instead of $\Spec(R)_{\rm vs}.$ The category $X_{\rm vs}$ has products (the usual reduction to the affine case combined with the tensor product of vertex rings of Section \ref{ES}).
 
 Let $Y = (Y, \cO_Y)$ be in $X_{\rm vs}$. The functor of points $h_Y \colon X_{\rm vs} \to \set$ of $Y$ is defined by $Z = (Z, \cO_Z) \mapsto {\rm Hom}_{X_{\rm vs}}(Z , Y).$ If $Z = \Spec(S)$ is affine, we write $Y(S)$ instead of $h_Y(\Spec(S)).$ If $Y = \Spec(B)$ is affine we denote $h_Y$ by $h^B.$ 
 
 A {\it vertex group scheme over} $X,$ or {\it an} $X-${\it vertex group}, is an object $\BG \in X_{\rm vs}$ whose functor of points $h_{\BG}$ is a group functor: Each $\BG(Y)$ has a group structure and this structure is functorial on $Y.$\footnote{\,We use boldface characters for group schemes to distinguish them from abstract groups.}   
 
 Let $R$ be a commutative vertex ring, and consider an  extension $B$ of $R.$ Assume that the functor of points $h^B$ factors through $\grp,$ that is, it is a group functor. Then the affine vertex scheme $\Spec(B)$ is an {\it affine $R-$vertex group scheme}. If we denote it by $\BH$, then we denote $B$ by $R[\BH].$ The comultiplication $\Delta \colon R[\BH] \to R[\BH] \otimes_R R[\BH]$ corresponding to the multiplication of $\BH$ according to Yoneda's correspondence, is a commutative $R-$vertex algebra homomorphism. Similarly for the counit and antipode.
 
  \begin{ex} $\BG^\infty$ (see Lemma \ref{jetlem}) is an affine vertex $R-$group. If $A \in \Ralg$ represents $\BG$ then $R[\BG^\infty] = A^\infty.$
\end{ex}

 Let  $(X, \mathcal O_X)$ be a vertex scheme. An\emph{ $\mathcal O_X-$vertex algebra} is a sheaf of vertex rings $\mathcal V$ on $X$, together with a morphism $\iota \colon \mathcal O_X \to \mathcal V$ of vertex ring sheaves, such that $\iota_U( \mathcal O_X(U)) \subset C(\mathcal V(U))$ for every open $U \subset X$.
 \begin{ex} \label{tildeV} A vertex algebra $V$ over a commutative vertex ring $(R, \uD)$ as defined in Section \ref{vcvr} yields a unique $\mathcal O_{\uD}-$vertex algebra $\widetilde{V}$ over $(\Spec (R), \cO_{\uD})$ satisfying $\widetilde{V}(D(f)) = V \otimes_R R_f.$   In the present paper our base will be mostly affine. We shall return to this general setting in future work. \end{ex}

\subsubsection{Vacuum subgroups} \label{vacuumsubgroups}

Let $k$ be a commutative unital ring that we view as a commutative vertex ring with trivial Hasse-Schmidt derivation.
In this case, any $R \in \kcva$ with Hasse-Schmidt derivation $\underline D = (D_i)_{i \geq 0}$ gives rise to a commutative $k-$algebra 
\[ K(R) = \ker(\underline D) = \{ r \in R \mid D_i(r) = 0 \text{ for all $i > 0$} \}. \]
This is the largest vertex subalgebra of $R$ with trivial Hasse-Schmidt derivation; it is called \textit{vacuum subalgebra} in \cite{Kac}  or \textit{ring of constants} in \cite{B}.

Consider an affine  $k-$group $\mathbf G\colon \kalg \to \grp.$ We attach to $\BG$ a group functor
$\GVac \colon \kcva \rightarrow \grp$ by defining $\GVac (S) = \BG(K(S))$ for any commutative $k-$vertex algebra $S = (S, \underline{D}).$}
\begin{prop} Let $\BG$ be an affine $k-$group. Then $\GVac$ is an affine vertex $k-$group.
\end{prop}
\begin{proof}
Suppose $A \in \kalg$ represents $\BG$.
	One can verify from the definition that  $\GVac$ is affine with coordinate ring $k[\GVac] = A,$  viewed as an object of $\kcva$ with trivial derivation. 
\end{proof}

\subsection{Representability of $\BAut(V)$}

 Let $V$ be a vertex algebra over a commutative vertex ring $R.$ We shall see in Section \ref{torsor} that $\BAut(V)$ is a sheaf of groups on the flat site of $R.$ In fact, we have that $\BAut(V)$ is an affine $R-$vertex group in important cases.

\subsubsection{Finitely generated vertex algebras} A classical result by Dong and Griess \cite{DongGr} says that the automorphism group of a finitely generated vertex operator algebra over $\mathbb C$ is an algebraic group. Their delicate argument can be adjusted to our situation to provide the following.

\begin{thm} \label{filtrep}
	Let $V$ be a filtered $R-$vertex algebra. Assume that the filtered pieces have the following properties.
\begin{enumerate}
    \item There exists $n \geq 0$ such that $U = F_n V$ generates $V$ as an $R-$vertex algebra,
    \item $U$ is a finite rank free $R-$module with basis $B = B_n$,
    \item for every $m > n$, there exists a set $E_m$ of $B-$admissible pairs $(\nu,b)$ (see {\rm \cite{DongGr}}) such that $F_m V$ is free with basis $B_m = \{ \nu(b) \}_{(\nu,b) \in E_m}$; for $m = n$ we put $E_n = \{(\id_V, b)\}_{b \in B}$,
    \item and we have $E_m \subset E_{m+1}$ for every $m \geq n$.
\end{enumerate}
   Then, $\BAut(V)$ is representable by a finitely generated commutative $R-$vertex algebra. \qed
\end{thm}

\subsubsection{Heisenberg and universal affine vertex algebras}\label{Secuniaff}

Let $k$ be a field of characteristic $0.$ Let us first recall the construction of the vertex algebras $V(\mathfrak g, \ell)$. See Section 6.2 of \cite{LL} for details.
   We are given a $k-$Lie algebra $\mathfrak g$ together with a nondegenerate symmetric invariant bilinear form $(\cdot \,, \cdot)$ on it and a scalar $\ell \in k$.
    We form the centrally extended loop algebra $\widehat {\mathfrak g} = (\mathfrak g \otimes_k k[t^{\pm 1}]) \oplus k\mathbf c$ with Lie bracket 
    \[ [a_{(m)}, b_{(n)}] = m (a,b) \delta_{m,-n} \mathbf c + [ab]_{(m+n)}, \quad [\mathbf c, \widehat {\mathfrak g}] = 0,\]
    where $a,b \in \mathfrak g$, $m,n \in \mathbb Z$, and where we use the notation $a_{(m)} = a \otimes t^m$.
    Denote by $k_\ell$ the one-dimensional $\mathfrak g[t] \oplus k \mathbf c-$module on which $\mathfrak g[t]$ acts trivially and $\mathbf c$ acts as multiplication by $\ell$, and form the induced $\widehat{\mathfrak g}-$module 
    \[V(\mathfrak g, \ell) = U(\widehat{\mathfrak g}) \otimes_{U(\mathfrak g[t] \oplus k \mathbf c)} k_\ell.\]
    One then defines $\mathbf 1 = 1_{U(\widehat{\mathfrak g})} \otimes 1_k$ and identifies $\mathfrak g$ as a subspace of $V(\mathfrak g, \ell)$ via $a \mapsto a_{(-1)} \mathbf 1$.
    Then there exists a unique vertex algebra structure on $V(\mathfrak g, \ell)$ such that $\mathbf 1$ is the vacuum vector and $Y(a,z) = \sum_{n \in \mathbb Z} a_{(n)} z^{-n-1}$ for all $a \in \mathfrak g$.
    We call $V(\mathfrak g, \ell)$ the \emph{universal affine vertex algebra on $\mathfrak g$ of level $\ell$}.

    For an abelian Lie algebra $\mathfrak h$ of finite dimension $n$ and $\ell \neq 0$ with $\sqrt\ell \in k$, we obtain the \emph{rank $n$ Heisenberg vertex algebra} $V(\mathfrak h, \ell) \simeq V(\mathfrak h, 1)$.
One associates to such $\mathfrak h$ the orthogonal $k-$group $\mathbf O(\mathfrak h)$ whose  functor of points associates to each ordinary $k-$algebra $R$ the group $O(\mathfrak h \otimes_k R)$ of $R-$linear automorphisms of $\mathfrak h \otimes_k R$ which preserve the bilinear form $(x \otimes a, y \otimes b)= (x,y) \otimes ab$, $x,y \in \mathfrak h$, $a, b \in R$.
  
\begin{thm}\label{affheis} Let $k$ be a field of characteristic $0$.
    \begin{enumerate}[label = \textup{\alph*)}]
    \item Let $\mathfrak h$ be a $k-$vector space of finite dimension $n$ equipped with a nondegenerate symmetric bilinear form, and let $V(\mathfrak h,1)$ be the associated Heisenberg $k-$vertex algebra. Then 
    \[ \BAut(V(\mathfrak h, 1)) \simeq (\mathbf G_a^n)^\infty \rtimes \mathbf O(\mathfrak h)^{\rm vac}, \]
		    where  $\mathbf G_a$ is the additive $k-$group, and for all $R \in \kcva$ the action of $\mathbf{O}(\mathfrak h)^{\rm vac}(R)$ on $\mathbf G_a^n (R) = R^n \simeq \mathfrak h \otimes_k R$ is the canonical one, i.e.\ it factors through $\mathbf{GL}_n(R)$.
    \item Let $\mathfrak g$ be a simple $k-$Lie algebra of finite dimension, $\ell \in k$, and $(\cdot \,, \cdot)$ a nonzero multiple of the Killing form.
 Let $V(\mathfrak g, \ell)$ be the associated universal affine vertex algebra.
    Then 
    \[ \BAut(V(\mathfrak g, \ell)) \simeq \BAut(\mathfrak g)^\infty,\]
    where $\BAut(\mathfrak g)(R)$ is the group of $R-$Lie automorphisms of $\mathfrak g \otimes_k R$.
    \end{enumerate}
	In particular $\BAut(V(\mathfrak h, 1))$ and $\BAut(V(\mathfrak g, \ell))$ are affine $k-$vertex group schemes (cf.\ Theorem \ref{filtrep}).
\end{thm}

The proof will be given after we recall and establish some preliminary results.   In general, $V(\mathfrak g, \ell)$ enjoys the following properties:
    \begin{enumerate}[label = \textup{\alph*)}]
        \item We have a decomposition $V(\mathfrak g, \ell) = \bigoplus_{n \geq 0} V(\mathfrak g, \ell)_n$ into finite dimensional $k-$spaces $V(\mathfrak g, \ell)_n$ spanned by all vectors of the form 
        $a^1_{-m_1} \cdots a^r_{-m_r} \mathbf 1$, where $r\in \mathbb Z_{\geq 0}$, $a^i \in \mathfrak g$, and $m_i \in \mathbb Z_{\geq 1}$  with  $m_1 + \cdots + m_r = n$. 
        This makes $V(\mathfrak g, \ell)$ into a graded vertex algebra, but in our context it is natural to work with the associated filtration $F_i V(\mathfrak g, \ell) = \bigoplus_{n \leq i} V(\mathfrak g, \ell)_n$.
\item The map $\mathfrak g \to V(\mathfrak g, \ell)$, $a \mapsto a_{-1} \mathbf 1$, yields an isomorphism of Lie algebras $\mathfrak g \simeq \mathfrak L(V(\mathfrak g,\ell))$.
	\item \label{pc} The form from Lemma \ref{comm} is given by $\sigma_{V(\mathfrak g, \ell)} = \ell(\cdot \,, \cdot)$.
	\item \label{pd} For any $\widehat {\mathfrak g}-$module $W $ that admits $w_+ \in W$ such that $a_{(n)} w_+ = 0$ for all $n \geq 0$ and $a \in {\mathfrak g}$ and whereon $\mathbf c$ acts as multiplication by $\ell$, there exists a unique $\widehat {\mathfrak g}-$map $V (\mathfrak g, \ell) \to W$ with $\mathbf 1 \mapsto w_+$.
    \end{enumerate}
    
We abstract from \ref{pc} and \ref{pd} a notion which is useful for the computation of filtered automorphism groups.

    Let $R$ be a commutative vertex ring and $V$ a filtered $R-$vertex algebra.
    We say that $V$ is \textit{$F_1-$universal} if for any filtered $R-$vertex algebra $W$ and any $R-$linear map $\Phi \colon F_1 V \to F_1 W$ respecting the $i-$th products for $i \geq 0$, and mapping $\mathbf 1 \mapsto \mathbf 1$, there exists a unique homomorphism $V \to W$ of filtered $R-$vertex algebras (co-)restricting to $\Phi$:
    \[ \begin{tikzcd} F_1 V \ar[r, "\forall \Phi"] \ar[d, phantom, "\subset", sloped] & F_1 W\ar[d, phantom, "\subset", sloped]  \\ V \ar[r, dashed, "\exists !"] & W. \end{tikzcd} \]

\begin{rem}\label{F1basechange}
    $F_1-$universality is preserved by base change.
\end{rem}

\begin{lem}\label{uniaffisF1}
    The universal affine $k-$vertex algebra $V(\mathfrak g, \ell)$ is $F_1-$universal.
\end{lem}

\begin{proof}
    Let $W$ be a filtered $k-$vertex algebra, and let $\Phi \colon F_1 V(\mathfrak g, \ell) \to F_1 W$.
   One shows that 
    \[ \rho \colon a_{(m)} \mapsto \Phi(a)_m, \quad \mathbf c \mapsto \ell \id_W,\]
    defines a $\widehat {\mathfrak g}-$structure on $W$.
    We thus obtain a unique $\widehat {\mathfrak g}-$module map $\Phi \colon V \to W$ extending the given map $F_1 V(\mathfrak g, \ell) \to F_1 W$.
    That it is a vertex algebra map can be checked on a generating set \cite[\@5.7.9]{LL}.
    For any $v \in V(\mathfrak g, \ell)$ and any $a \in V(\mathfrak g, \ell)_1 \simeq \mathfrak g$, we have
    $\Phi(a_n v) = \Phi(a_{(n)} v) = a_{(n)} \Phi(v) = \Phi(a)_n \Phi(v). $
   That $\Phi$ preserves the filtration follows from the fact that the filtered pieces $F_n V(\mathfrak g, \ell)$ are spanned by elements of the form $a^1_{-m_1} \cdots a^r_{-m_r} \mathbf 1$ with $a^i \in F_1 V(\mathfrak g, \ell)$ and $m_1 + \cdots + m_r \leq n$.
    Thus, $V(\mathfrak g, \ell)$ is $F_1-$universal.
\end{proof}

\begin{lem}\label{isoF1res}
    Let $V$ be an $F_1-$universal filtered $R-$vertex algebra such that $R \simeq F_0 V$ and $F_{-1} V = 0$. 
    Then the restriction map identifies $\Aut_{\Rfva}(V)$ with the group $\Aut_R(F_1 V, \sigma_V)$ of $R-$linear bijections that preserve the $0$-th product, the bilinear form $\sigma_V$, and the vacuum vector $\mathbf 1$.
\end{lem}

\begin{proof}
    The claim follows by definition of $F_1-$universality since the $i-$th products are zero on $F_1 V$ for $i \geq 2$ when $F_{-1} V= 0$.
\end{proof}

\begin{lem} \label{F1}
    Let $V$ be a filtered $R-$vertex algebra with $R \simeq F_0 V$ and $F_{-1} V = 0$.
    Denote by $\Aut_R(F_1V)$ the group of $R-$linear bijections that preserve the $0$-th product and the vacuum vector $\mathbf 1$. 
    Consider $\mathfrak L(V)^\ast = \Hom_R(\mathfrak L(V),R)$ as an abelian group.
    Finally, suppose the exact sequence of $R-$modules
    \[ \begin{tikzcd} 0 \ar[r] & R \mathbf 1 \ar[r, "\subset"] & F_1 V \ar[r, "p"] & \mathfrak L(V) \ar[l, dashed, bend left = 60, "j"] \ar[r] & 0 \end{tikzcd} \]
    is split, and define $\alpha \colon \mathfrak L(V) \times \mathfrak L(V) \to R$ via $\alpha(x,y) \mathbf 1 = j(x)_0 j(y) - j [xy]$.
    Then:
    \begin{enumerate}[label = \textup{\alph*)}]
	\item \label{F1a} The map
	\begin{align*} \Aut_R(F_1 V) &\to \mathfrak L(V)^\ast \rtimes \Aut_{\RLie}(\mathfrak L(V)), \\
    \Phi &\mapsto (\Phi j \overline \Phi^{\,-1} \!-j, \overline \Phi), \quad \overline \Phi = p \Phi j, \end{align*}
	is a group homomorphism into the semidirect product with multiplication $(g, \psi) (f, \varphi) = (g + f \psi^{-1}, \psi \varphi)$,
    where $f,g \in \mathfrak L(V)^\ast$ and $\varphi, \psi \in \Aut(\mathfrak L(V))$.
    \end{enumerate}
    Furthermore, under this map,
    \begin{enumerate}[resume, label = \textup{\alph*)}]
	\item \label{F1b}  $\Aut_R(F_1 V)$ is identified with the subgroup of pairs $(f, \varphi)$ satisfying
	\begin{equation}\label{semisub} \alpha(x,y) - \alpha(\varphi^{-1} x, \varphi^{-1} y) = f[xy],\, \text{and} \end{equation}
	\item \label{F1c}  $\Aut_R(F_1 V, \sigma_V)$ is identified with the subgroup of pairs $(f, \varphi)$ where \eqref{semisub} holds and $\varphi$ preserves $\sigma_V$.
    \end{enumerate}
\end{lem}

\begin{proof} These are straightforward calculations when taking into account that $ R\mathbf 1 = F_0 V$ annihilates the $0$-th product. 
\end{proof}

\begin{proof}[Proof of Theorem \ref{affheis}] Let $V =  V(\mathfrak g, \ell)$ and define $j \colon \mathfrak L(V_R) \simeq \mathfrak g \otimes_k R \subset F_1V$ to be the section induced by $\mathfrak g \simeq V(\mathfrak g, \ell)_1$. From the last lemma we obtain that:
\begin{enumerate}[label = \textup{\alph*)}]
        \item \label{F1unia} If $\mathfrak g = \mathfrak h$ is abelian and $\ell \neq 0$ then \[ \Aut_R(F_1(V_R), \sigma_V) \simeq R^n \rtimes O(\mathfrak h \otimes_k \ker D),\]
        where $D = D_1 \colon R \to R$ is the canonical derivation.
        \item \label{F1unib} If $\mathfrak g$ is simple then \[ \Aut_R(F_1(V_R), \sigma_V) \simeq \Aut_{\RLie}(\mathfrak g \otimes_k R).\]
    \end{enumerate}

    The $R-$vertex algebra $V_R = V(\mathfrak g, \ell) \otimes_k R$ is $F_1-$universal by Lemmas \ref{uniaffisF1} and \ref{F1basechange}.
    Thus, Lemma \ref{isoF1res} applies and, together with a) and b) above, shows that we have the isomorphisms prescribed by the theorem ``on $R-$points.'' It remains to verify functoriality.
    For this one should note that under the identifications $F_1(V_R) \simeq (F_1 V(\mathfrak g, \ell)) \otimes_k R$ and $\mathfrak L(V_R) \simeq \mathfrak g \otimes_k R$ (Lemma \ref{Lie}) we can make both sides of the map 
    \begin{equation}\label{mapfrom46} \Aut_R(F_1 (V_R)) \to \mathfrak L(V_R)^\ast \rtimes \Aut_{\RLie}(\mathfrak L(V_R))\end{equation}
    from Lemma \ref{F1} into a functor on $R.$  Choosing the lift $j$ to be the $R-$extension of $\mathfrak g \to F_1V(\mathfrak g,\ell)$ makes \eqref{mapfrom46} into a natural transformation.
    
    Finally, in the case that $\mathfrak g = \mathfrak h$ is abelian,  the identification $\mathfrak L(V_R)^\ast \simeq \mathfrak h \otimes_k R$ by means of the bilinear form is functorial.
\end{proof}

\section{Descent theory} 
\subsection{Twisted forms of vertex algebras} \label{tfva}

Let $R$ be a commutative vertex ring, and $V$  a fixed  $R-$vertex algebra. If $S$ is an extension of $R$, we say that an $R-$vertex algebra $W$ {\it is an $S/R-$twisted form of}  $V,$ if the $S-$vertex algebras $W \otimes_R S$ and $V \otimes_R S$ are isomorphic. In this case we also say that $S$ {\it trivializes} $W.$

 If the trivializing $S$ above can be chosen to be a faithfully flat and finitely presented (fppf) extension of $R$ (as rings), then we say that $W$ is a {\it twisted form of} $V.$ If $S$ can be chosen to be finite and \'etale, we say that the twisted form under consideration is {\it isotrivial.}

The following is an extremely important example of isotrivial twisted forms, one that we will revisit in Section \ref{app}.

\subsection{Loop vertex algebras}\label{loopva}

    Let $k$ be a field of characteristic zero,
     $V$ be a $k-$vertex algebra, and let $g \in \Aut(V)$ be an automorphism of finite period $m \in \mathbb Z_{> 0}$.
    Assume that $k$ contains a primitive $m-$th root of unity $\zeta.$ Then $g$ is diagonalizable with eigenspaces $V^{g,r} = \{ v \in V \mid gv = \zeta^{-r} v \}$, $0 \leq r < m$.

    The ring $S_m = k[t^{\pm\frac{1}{m}}]$ carries the structure of a commutative $k-$vertex algebra defined by the derivation $D = \frac{d}{dt}$; the vertex operators are given by 
    \[Y(a(t),z)b(t) = \big( e^{z \frac{d}{dt}} a(t)\big) b(t) = a(t+z)b(t).\]
    The $k-$algebra homomorphism $\gamma \colon S_m \to S_m$ defined by $t^{\frac{1}{m}} \mapsto \zeta t^{\frac{1}{m}}$ commutes with $D$ and therefore is a $k-$vertex algebra automorphism.
    
    On the tensor product $V \otimes_k S_m$ we thus have the automorphism $g \otimes \gamma$ whose vertex subalgebra of fixed points is the \emph{twisted loop vertex algebra} \[L(V,g) = \bigoplus_{0 \leq r < m} V^{g,r} \otimes_{k} t^{\frac{r}{m}} k[t^{\pm 1}] .\]
    Observe that $\gamma$ restricts to the identity on $R = k[t^{\pm 1}] \subset S_m.$ Thus $g \otimes \gamma$ is $R-$linear and $L(V,g)$ is an $R-$vertex algebra.

    The inclusion $L(V,g) \subset V \otimes_k S_m$ of $R-$vertex algebras induces an $S_m-$map \[ L(V,g) \otimes_R S_m \to V \otimes_k S_m, \quad v \otimes q \otimes p \mapsto v \otimes pq. \] 
    This map is an isomorphism, hence $L(V,g)$ is a twisted form of $V \otimes_k R$ trivialized by $S_m.$
Since  $S_m/R$ is a finite \'etale (in fact, Galois) extension, $L(V,g)$ is isotrivial.

 \subsection{Faithfully flat descent of twisted forms}\label{ffdesc}

	Let $S/R$ be an extension of commutative vertex rings.
	Given a group functor $\mathbf G \colon \Rcva \to \grp$ we  define in the usual fashion \cite{Waterhouse} the set of {\it cocycles} $Z^1(S/R, \mathbf G)$ and the corresponding  \textit{non-abelian cohomology set} $H^1(S/R, \mathbf G);$ a pointed set whose distinguished element we denote by $1.$ Like in the classical theory {\it loc.\ cit.}, see also \cite{Car}, we have:

\begin{thm}\label{descent}
    Let $S/R$ be an extension of vertex rings such that $S$ is faithfully flat as an $R-$module (cf.\@ Lemma \ref{charRva}), and let $V$ be an $R-$vertex algebra.
    Then we have a one-to-one correspondence between the $R-$isomorphism classes of $S/R-$forms of $V$ and the set $H^1(S/R, \BAut(V))$.
    The class of $V$ corresponds to the distinguished element $1$.
\end{thm}

    \subsection{Galois cohomology}

Let $S/R$ be an extension of commutative vertex rings, and assume that $S/R$ is (finite)  Galois with Galois group $\Gamma$ (see \cite[\@5.6--7]{Knus} and also \cite[\@1.3]{CHR} for a list of equivalent characterizations of Galois extension of rings). Relevant to us is that $S/R$ is faithfully flat.
    \begin{ex}
    The extension $S_m/R$ from Section \ref{loopva} is Galois with Galois group $\mathbb Z/m \mathbb Z$ where the generator $1$ acts via $t^{\frac{1}{m}} \mapsto \zeta t^{\frac{1}{m}}$, for $\zeta$ a fixed primitive $m-$th root of unity.
\end{ex}

For a filtered $R-$vertex algebra $V,$  the Galois group $\Gamma$ acts on  $\BAut(V)(S)$ by group automorphisms as follows: If $h \in \BAut(V)(S)$ and $\gamma \in \Gamma$, then 
\[ {}^\gamma h = (\id \otimes \gamma) h (\id \otimes \gamma^{-1}).\] 
This action allows us to define the Galois cohomology pointed set  $H^1(\Gamma, \BAut(V)(S))$ which, in turn, agrees with the faithfully flat cohomology defined in Section \ref{ffdesc} (see \cite{Waterhouse} or \cite{Knus}): 
\[ H^1 (S/R, \BAut(V)) \simeq H^1(\Gamma, \BAut(V)(S)). \]

\subsection{$\widehat S/R-$forms and continuous cohomology} \label{H1hat}

Assume $k$ is of characteristic $0$ and contains all primitive roots of unity. As in \cite{KLP} we want to compare $R-$isomorphism classes of twisted loop algebras (see Section \ref{loopva}) all at once.     If we set \[\widehat S = \textstyle\varinjlim_{m} S_m = k[t^q \mid q \in \mathbb Q]\] then any loop algebra $L(V,g)$ is visibly an $\widehat S/R-$twisted form of $V \otimes_k R.$ 
      
The extension $\widehat S/R$ is a limit of Galois extensions. A suitable finiteness condition on $V$ allows us to compute $\widehat S/R-$forms in terms of \emph{continuous group cohomology} $H^1_{\operatorname{ct}}$. 

\begin{thm} Let $V$ be a finitely generated  $R-$vertex algebra. Then.
\begin{enumerate}[label = \textup{\alph*)}]
	\item \label{ct1} The action of the profinite group $\widehat{\mathbb Z}$ on the discrete group $\BAut(V)(\widehat S)$ is continuous.
        
        \item \label{ctH1}  We have a bijection of pointed sets $H^1(\widehat S/R, \BAut(V)) \simeq H^1_{\operatorname{ct}} (\widehat {\mathbb Z}, \BAut(V)(\widehat S)).$   \qed \end{enumerate}  \end{thm} 
Here $\widehat{\mathbb Z} := \langle \hat{1} \rangle$ denotes the profinite completion of $\Z,$ and its action on $\widehat S$ is induced from the Galois actions of $\mathbb Z/m \mathbb Z$ on $S_m/R$ so that $\hat{1} \colon t^{\frac{1}{m}} \mapsto \zeta_m t^{\frac{1}{m}}$, where we now choose our primitive $m-$th root of unity  $\zeta_m$ to be compatible, namely such that $\zeta_{m \ell}^\ell = \zeta_m$ holds for all $m, \ell \in \mathbb Z_{\geq 1}$.

\subsection {Sheaves} \label{sheaf}
Let us begin by recalling some definitions (see \cite{DeGa} Ch.III \S1 for details) adjusted to our setting. Let $R$ be a commutative vertex ring and consider the category $\Rcva.$ Let $S$ be a commutative $R-$vertex algebra; i.e.\@ an object of $\Rcva.$ A {\it covering (or cover) of} $S$ is a commutative $S-$vertex algebra $S'$ that is faithfully flat and finitely presented (as a ring extension). The category $\Rcva$ together with the family of coverings defined above is a {\it site} \cite[Tag 00VH]{SP} which we call the {\it flat site of} $R$ and denote by $\Rcva_{\rm fl}.$
\medskip

A functor $\BF \colon \Rcva \to \set$ is an {\it  sheaf on} $\Rcva_{\rm fl}$ if   conditions \ref{f1} and \ref{f2} below hold.
\begin{enumerate}[label = {\textup{(F\arabic*)}}]
\item \label{f1} The canonical map $\BF(\prod_{i = 1}^n R_i) \to \prod_{i = 1}^{n} \BF(R_i)$ is bijective for all $R_1, \dots , R_n$ in $\Rcva.$
\item \label{f2} For all $S$ in $\Rcva$, and for all covers $S'$ of $S$  the sequence \[\begin{tikzcd}
 \BF(S) \ar[r] & \BF(S') \ar[r,shift left=.75ex]
  \ar[r,shift right=.75ex,swap] & \BF(S'\otimes_S S')
  \end{tikzcd}\]
 is exact.
 \end{enumerate}


\begin{lem}\label{Autsheaf} Let $R$ be a commutative vertex ring. 
	\begin{enumerate}[label = \textup{\alph*)}]
\item If $V$ and $W$ are filtered $R-$vertex algebras, then $\BHom(V,W)$ and $\BAut(V)$ are group sheaves on $\Rcva_{\rm fl}.$

\item If $X$ is a vertex scheme, its functor of points $h_X$ is a sheaf on $\Rcva_{\rm fl}.$
\item For a filtered $R-$vertex algebra $V$, the functor $V_a$ that attaches to an $S \in \Rcva$ the $S-$vertex algebra $V \otimes_R S$ is a sheaf of vertex algebras on $\Rcva_{\rm fl}$. See \cite[\@I Prop.\ 4.6.2]{SGA3} for the relation between $V_a$ and $\widetilde V$ from Example \ref{tildeV}.
	\qed
\end{enumerate}
\end{lem}

\subsection{Torsors}\label{torsor}
Let $\BH$ be a group sheaf over $R.$ An $\BH-${\it sheaf torsor over $R$} is an $R-$sheaf of sets $\mathbf E$ together with a right action of $\BH$ for which there exists a covering $S$ of $R$ such that
\begin{equation}\label{tor1}
\mathbf E|_{S} \simeq \BH|_{S}.
\end{equation}
Here and elsewhere for all $S \in \Rcva$ we denote by $\mathbf F|_S$  the restriction of $\mathbf F$ to the flat site of $\Scva.$  The isomorphism \eqref{tor1} is one of sheaves with $\BH-$action. In this situation we say that $S$ {\it trivializes} $\mathbf E.$
 The  {\it trivial sheaf torsor} is the group sheaf $\BH$  acting on itself by right multiplication. We will denote by $\tH^1(R,\BH)$ the set of isomorphism classes of $\BH-$sheaf torsors over $R.$ This is a pointed set with the class of the trivial sheaf torsor as distinguished element.

We now recall how the non-abelian cohomology defined in Section \ref{ffdesc} can be used to classify $\BH-$torsors.
\begin{prop}\label{S/R} Let $S$ be a cover of $R.$ There exists a bijection  of pointed sets between $H^1(S/R, \BH)$ and the subset of $\tH^1(R, \BH)$ consisting of  isomorphism classes of $\BH-$sheaf torsors over $R$ that are trivialized by $S.$
\end{prop}
	\begin{proof} The reasoning follows along traditional lines,  see \@ \cite[\@III Prop.\@ 4.6]{Mln} and \cite[\@III \S4 Theo.\@ 6.4]{DeGa}. The only delicate point is how to construct a sheaf torsor out of a cocycle $h \in Z^1(S/R, \BH).$ We use $h \in \BH(S \otimes_R S)$ to glue two copies of the sheaf $\BH|_S$, cf.\@ \cite[Tag 04TR]{SP}.
\end{proof}


If $S$ and $T$ are covers of $R,$ then $S \otimes_R T$ is a  cover of $S,T$ and $R.$ We can thus consider the direct limit of the pointed sets $H^1(S/R, \BH)$ over all covers $S$ of $R.$\footnote{\,There are set theoretical matters to take into account: The family of all fppf covers of $R$ is not a set, see \cite[Tag 00VI]{SP}. In our case this is simple to overcome.}  Set
\begin{equation}\label{tor2}
 H^1(R, \BH) := \varinjlim_S H^1(S/R, \BH) .
\end{equation}
From the fact that every sheaf $\BH-$torsor is trivialized by a  cover of $R$ we obtain
\begin{prop}\label{limitdescent}
There exists a bijection of pointed sets $H^1(R, \BH) \simeq \tH^1(R,\BH)$ compatible with the bijections of Proposition \ref{S/R}. \qed
\end{prop}

\begin{prop}\label{tor3} Let $\BG$ be an affine $R-$group scheme and consider its corresponding vertex $R-$group $\BG^\infty.$ If $\BG$ is smooth the canonical map $\rho : H^1(R, \BG^\infty) \to H^1(R, \BG)$ is bijective.
       \end{prop}
\begin{proof} Let $S$ be a cover of $R$ in $\Rcva_{\rm fl}$. Then, ignoring the Hasse-Schmidt derivations, $S$ is a cover of $R$ in the flat site $R_{\rm fl}$ of $R.$ Since  $\BG^\infty(S) = \BG(S)$ we have $H^1(S/R, \BG^\infty) = H^1(S/R, \BG).$ Passing to the limit yields the map $\rho$ of the Proposition.  

 In view of Proposition \ref{limitdescent} we obtain a sequence of pointed sets maps
$$\tH^1(R, \BG^\infty) = H^1(R, \BG^\infty) \to H^1(R, \BG) = \tH^1(R, \BG).$$

We do not know whether the middle map $\rho \colon H^1(R, \BG_\infty) \to H^1(R, \BG)$ is surjective. Indeed, a priori there could be a $\BG-$sheaf torsor over $R$ that is trivialized by some  cover $S$ of $R$, but no such $S$ exists that it is also in $\Rcva.$  

Let $[\mathbf{E}] \in H^1(R, \BG).$ Because $\BH$ is smooth, $\mathbf{E}$ can be trivialized by an \'etale extension $S$ of $R$ \cite[\@III \S4 Prop.\ 4.1 infra]{Mln}. Since the iterative Hasse-Schmidt derivation of $R$  extends (uniquely) to an iterative Hasse-Schmidt derivation of $S$ (use \cite[Theo.\@ 3.6]{Voj} to lift the derivation, then argue that the lift is iterative),  we can  view $S$ as a cover of $R$ in $\Rcva_{\rm fl}.$  Injectivity of $\rho$ is proved using similar ideas. \end{proof}

\section{Applications to twisted forms of vertex algebras over Laurent polynomials} \label{app}

Let $k$ be an algebraically closed field of characteristic 0, and consider the rings $R = k[t^{\pm 1}]$ and $\widehat S = k[t^q \mid q \in \mathbb Q]$ with their vertex ring structure induced by the derivation $D = \frac{d}{dt}$.  The foregoing material shows that non-abelian cohomology provides a tool for computing and classifying forms. We will explicitly put this to good use in two important cases.

\subsection{Twisted forms of affine vertex algebras}
Consider  a simple $k-$Lie algebra $\mathfrak g$, a scalar $\ell \in k$, and let $V = V(\mathfrak g, \ell)$.
\begin{thm} Let $W$ be a twisted form of $V(\fg, \ell) \otimes_k R$ (see Section \ref{tfva}). Then $W \simeq L(V(\fg, \ell), \sigma)$ for some (unique up to conjugacy) diagram automorphism $\sigma$ of the Dynkin diagram of $\fg$ (see Section \ref{loopva}). In particular, $W$ is isotrivial.
\end{thm} 
\begin{proof}
	Let $S$ be an fppf cover of $R$ trivializing $W$.
Theorem \ref{descent} shows that $W$ corresponds to an element $[W] \in H^1( S/R, \BAut(V_R))$, which we view as an element of $H^1(R, \BAut(V_R)) := \varinjlim_S H^1(S/R, \BAut(V_R).$  Thus defined,  $H^1(R, \BAut(V_R))$ classifies the isomorphism classes of vertex algebras over $R$ that become $V(\fg, \ell) \otimes_k S$  for some faithfully flat and finitely presented $S \in \Rcva.$

Theorem \ref{affheis} and Proposition \ref{tor3} yield
\begin{equation}\label{VALie}
 H^1(R, \BAut(V_R)) = H^1(R, \BAut(\mathfrak g_R)^\infty) = H^1(R, \BAut(\mathfrak g_R))
 \end{equation}
 
 The main theorem of \cite{Pianzola2} shows that $H^1(R, \BAut(\mathfrak g_R)) \simeq H^1_{\operatorname{ct}}(\widehat {\mathbb Z}, \Out(\fg)),$ where $\Out(\fg)$  is the group of automorphisms  of the Dynkin diagram of $\fg.$ Since $\widehat{\mathbb Z}$ acts trivially on $\Out(\fg)$, the latter is identified with the conjugacy classes of $\Out(\fg)$. In term of Lie algebras this yields that the only twisted forms of $\fg_R$ are the loop algebras $L(\fg, \sigma)$ that one encounters in the theory of affine Kac-Moody Lie algebras. Here $\sigma \in \Out(\fg)$ is viewed as an automorphism of $\fg$ via the choice of a Killing couple  of $\fg.$ By thinking in terms of cocycles, it is clear that  the isomorphism class corresponding to $L(V, \sigma)$ in $H^1(R, \BAut(V_R))$ maps to that of $L(\fg, \sigma)$ in $H^1(R, \BAut(\fg_R)).$  Now (\ref{VALie}) shows that the  $L(V, \sigma)$ are up to isomorphism all the twisted forms of $V_R.$
\end{proof}


Consider next the simple quotient algebras $L(\mathfrak g, \ell)$ with $\ell \neq -h^\vee$ where  $h^\vee$ is the dual Coxeter number of $\mathfrak g$.  That is, if we denote by $I \subset V = V(\mathfrak g,\ell)$ the unique maximal ideal, then $L(\mathfrak g, \ell) = V/I$. Here we have normalized the invariant form on $\fg$ such that long roots have squared length $2$.
Unlike the situation of the last theorem, we do not know if any twisted form of $L(\mathfrak g, \ell) \otimes_k R$ is isotrivial, i.e.\@ trivialized by a finite \'etale extension of $R$, hence by $\hS.$  We nonetheless have the following.

\begin{thm} \label{Lforms}
    The $R-$isomorphism classes of filtered $\widehat S/R-$twisted forms of the simple affine vertex algebra $L(\mathfrak g, \ell) \otimes_k R$ are in bijective correspondence with the conjugacy classes of the group $\Out(\fg)$. \qed
\end{thm}

\subsection{Isotrivial forms of the Heisenberg vertex algebras}
Recall that $\mathfrak h$ is a $k-$vector space of finite dimension $n$ equipped with a nondegenerate symmetric bilinear form, and  $V(\mathfrak h,1)$ is the associated Heisenberg $k-$vertex algebra. We have seen that
    $$ \BAut(V(\mathfrak h, 1)) \simeq (\mathbf G_a^n)^\infty \rtimes \mathbf O(\mathfrak h)^{\rm vac}. $$
    
 The loop algebras $L(V(\mathfrak h,1), \sigma)$ are examples of isotrivial forms of $V(\mathfrak h,1)_R.$ The following result  classifies all such forms.
 
 \begin{thm}\label{Hforms} 
 The $R-$isomorphism classes of filtered $\widehat S/R-$twisted forms of the Heisenberg vertex algebra $V(\mathfrak h,1)_R$ are in bijective correspondence with the conjugacy classes of elements of finite order of the group $\mathbf O(\mathfrak h)(k)$. Any such twisted form is isotrivial.
\end{thm}

 \begin{proof}   In Section \ref{H1hat} we have shown that the isomorphism classes of forms of $V(\mathfrak h,1)_R$ that are trivialized by  $\hS$ are in bijection with the pointed set $ H^1_{\operatorname{ct}} (\widehat {\mathbb Z}, \BAut(V(\mathfrak h, 1)(\widehat S)).$ Moreover, by continuity any such form is in fact trivialized by some $S_m,$ hence isotrivial. 
 
 In what follows we will denote by $F$ the group $ \mathbf O(\mathfrak h)(k)$ with trivial action of $\hZ,$  and the $R-$group $\mathbf G_{a, R}^n$ by $\BG.$ Note that $H^1(\hZ, \mathbf O(\mathfrak h)^{\rm vac}(\hS)) = H^1(\hZ, F),$ and  that the action of $F$ on $\BG(\hS)$ is given by elements of $\BGL_n(k).$ Our relevant exact sequence of $H^1$ reads
     \begin{equation}\label{Hcoho2} 
  H^1_{\operatorname{ct}}(\hZ, \mathbf G(\hS)) \rightarrow H^1_{\operatorname{ct}}(\hZ, \mathbf G(\hS) \rtimes F) \rightarrow H^1_{\operatorname{ct}}(\hZ, F).   \end{equation}
  Since $H^1_{\operatorname{ct}}(\hZ, F)$ is in bijection with the conjugacy classes of elements of finite order of $\mathbf O(\mathfrak h)(k),$ our aim is to show that this last map is bijective. Surjectivity is clear. As for injectivity, let $u \in Z^1_{\operatorname{ct}}(\hZ, F). $ Recall that $u(\hat{1}) = \sigma$ is an element of finite order $m$ of $F.$  Following Serre \cite{Se} we construct the twisted $\hZ-$group  $\mathbf G(\hS)_u.$ The bijectivity of (\ref{Hcoho2}) that we are after would hold if we can show that for all $u$ as above the corresponding $H^1_{\operatorname{ct}}(\hZ, \mathbf G(\hS)_u)$ vanish. Now this last $H^1$ lies inside the isotrivial part of $H^1(R, \, {_u}\BG)$ where ${_u}\BG$ is the $R-$group constructed from $u$ by Galois descent. Since ${_u}\BG$ is a twisted form of $\BG$ and such forms  are classified (up to isomorphism) by $H^1(R, \BAut(\BG)) = H^1(R, \BGL_{n,R}) = 1,$ this last since every projective $R-$module of finite type  is free, we conclude that $\BG \simeq  {_u\BG}.$ Thus $H^1(R, {_u}\BG) = 1.$
  \end{proof}


\subsection{$R-$isomorphism versus $k-$isomorphism}

Let $U$ and $W$ be $R-$vertex algebras. Assume that $U$ and $W$ are isomorphic as $k-$vertex algebras. Then we may ask whether $U$ and $W$ are isomorphic as $R-$vertex algebras. 
In many cases, the structure map of an $R-$vertex algebra is in fact an isomorphism onto the centre. Suppose that this is so for $U$ and $W$. 
A $k-$isomorphism $\varphi \colon U \to W$ then induces a $k-$automorphism of $R$, but the automorphism group of $R = (k[t^{\pm 1}], \frac{d}{dt})$ is trivial. It follows that $\varphi$ is $R-$linear.

For example, if $V$ is a $k-$vertex algebra with $C(V) \simeq k$, then the loop vertex algebras of $V$ have centre $R$. We conclude that 
\[ L(V,g) \simeq_R L(V,h) \iff L(V,g) \simeq_k L(V,h) \]
for any two finite order automorphisms $g$ and $h$ of $V$.

\subsection{Pullback of twisted modules and a correspondence of Li} \label{twmodLi}

A classical result in Kac-Moody theory says that any twisted loop algebra associated to a finite order automorphism of a simple Lie algebra is in fact isomorphic to a twisted loop algebra coming from a Dynkin graph automorphism. In \cite{Li1, Li2}, Li shows through calculations that this principle carries over to twisted modules for the vertex operator algebras $L(\fg, \ell)$. This can be explained using Galois cohomology.

Fix a vertex algebra $V$ over an algebraically closed  field $k$ of characteristic 0. We shall use a shorthand notation for elementary tensors in the loop vertex algebra $L(V,g) = \bigoplus_{0 \leq r < M} V^{g,r} \otimes_k t^{\frac{r}{M}} k[t^{\pm 1}]$ associated to an automorphism $g \colon V \to V$ of period $M$ and a choice of primitive $M-$th root of unity $\zeta$.
Namely, we set $v_{(n)} = v \otimes t^n$, for $v \in V^{g,r}$ and $n \in \frac{r}{M} + \mathbb Z$.
We have the key formula
\begin{equation} \label{keyf} (u_{(m)})_\ell(v_{(n)}) = \sum_{i \geq 0} u_{\ell+i} v \otimes (t^m)_{-i-1} t^n = \sum_{i \geq 0} \binom{m}{i} (u_{\ell+i} v)_{(m+n-i)} \end{equation}
	for $u \in V^{g,r}$ and $v \in V^{g,s}$, $m \in \frac{r}{M} + \mathbb Z$ and $n \in \frac{s}{M} + \mathbb Z$, and $\ell \in \mathbb Z$.

The twisted loop vertex algebras $L(V,g)$ play a prominent role in the study of $g-$twisted $V-$modules as defined in \cite[Def.\@ 3.1]{DLM}.
Here is a concrete connection.

\begin{prop}\label{tfiso}
    Let $h \colon V \to V$ be another automorphism of finite period $M$. Let $\varphi \colon L(V,g) \to L(V,h)$ be a $k[t^{\pm1}]-$vertex algebra map, and let $W$ be a weak $h-$twisted $V-$module with vertex operator map $Y_W(v,z) = \sum_{n \in \frac{1}{M} \mathbb Z} v_n z^{-n-1}$. 
    Denote by $\overline \varphi$ the linear map $L(V,g) \to \End(W)$ obtained by composing $\varphi$ with $L(V,h) \to \End(W)$, $v_{(n)} \mapsto v_n$. 
    
    Then the linear map $(\varphi^\ast Y_W)(\cdot, z) \colon V \to \End(W)[[z^{\frac{1}{M}},z^{-\frac{1}{M}}]]$ defined by 
	   \[ (\varphi^\ast Y_W) (v,z) = \sum_{n \in \frac{r}{M} + \mathbb Z} \overline \varphi(v_{(n)}) z^{-n-1}, \quad v \in V^{g,r}, \]
    makes $W$ into a weak $g-$twisted $V-$module, which we denote simply by $\varphi^\ast W$.
\end{prop}

\begin{proof}
	For $u \in V^{g,r}$, $m \in \frac{r}{M} + \mathbb Z$, we write $\varphi(u_{(m)}) = \sum_j u^j_{(m_j)}$ for some $u^j \in V^{h, r_j}$ and $m_j \in \frac{r_j}{M} + \mathbb Z$, and observe then that $\varphi (u_{(m+\ell)}) = \sum_j u^j_{(m_j+\ell)}$ follows for any $\ell \in \mathbb Z$ from $k[t^{\pm 1}]-$linearity.
	The regularity axiom follows, and so does the vacuum axiom, since $\varphi(\mathbf 1_{(0)}\!) = \mathbf 1_{(0)}$.
	Verification of the Jacobi identity  for $\varphi^\ast Y_W$ additionally uses \eqref{keyf}, the twisted Jacobi identity for $Y_W$, and that $\varphi$ is a homomorphism, but is otherwise mechanical.
\end{proof}

\begin{rem}
    The above correspondence satisfies some basic categorical compatibilities. 
    For example, if $W_1 \to W_2$ is a map of weak $h-$twisted $V-$modules, then the same set map is a map of weak $g-$twisted $V-$modules $\varphi^\ast W_1 \to \varphi^\ast W_2$. 
    This makes $\varphi^\ast(-)$ into an additive functor from the category of weak $h-$twisted modules to the category of weak $g-$twisted modules.

    If $p \colon V \to V$ is another finite order automorphism, and $\psi \colon L(V,h) \to L(V,p)$ is another $k[t^{\pm 1}]-$vertex algebra map, then for any weak $p-$twisted $V-$module $U$, we have $(\psi \varphi)^\ast U = \varphi^\ast \psi^\ast U$.
    From this it follows that if $\varphi$ is an isomorphism, then $\varphi^\ast (-)$ is an isomorphism of categories.

	We  postpone further discussion of this correspondence to a future paper.
\end{rem}

\noindent Li's correspondence is corollary.

\begin{cor}[Cf.\@ Section 5 of \cite{Li1}]
    Let $\mathfrak g$ be a finite dimensional simple $k-$Lie algebra, and let $\ell \in k \setminus \{ -h^\vee\}$, where $h^\vee$ of $\fg$ is the dual Coxeter number of $\fg.$
    Let $g$ and $h$ be two finite order automorphisms of $\mathfrak g$ whose images in the outer automorphism group $\Out(\fg)$ are conjugate.

    Then, viewing $g$ and $h$ as automorphisms of the simple conformal $k-$vertex algebra $L(\fg, \ell)$, we have a functorial one-to-one correspondence between weak $g-$twisted and weak $h-$twisted modules.
\end{cor}

\begin{proof}
    This follows from Theorem \ref{Lforms} together with Proposition \ref{tfiso}.
\end{proof}

\bibliographystyle{amsplain}
\bibliography{refs}

@article{KLP,
shorthand = {KLP},
title = {Differential conformal superalgebras and their forms},
journal = {Advances in Mathematics},
volume = {222},
number = {3},
pages = {809-861},
year = {2009},
author = {V.\ Kac and M.\ Lau and A.\ Pianzola},
}

@incollection {Li1,
shorthand = {L1},
    AUTHOR = {Li, H.\ },
     TITLE = {Local systems of twisted vertex operators, vertex operator
              superalgebras and twisted modules},
 BOOKTITLE = {Moonshine, the {M}onster, and related topics ({S}outh
              {H}adley, {MA}, 1994)},
    SERIES = {Contemp. Math.},
    VOLUME = {193},
     PAGES = {203--236},
 PUBLISHER = {Amer. Math. Soc., Providence, RI},
      YEAR = {1996},
}

@article {Li2,
shorthand = {L2},
    AUTHOR = {Li, H.\ },
     TITLE = {Twisted modules and pseudo-endomorphisms},
   JOURNAL = {Algebra Colloq.},
  FJOURNAL = {Algebra Colloquium},
    VOLUME = {19},
      YEAR = {2012},
    NUMBER = {2},
     PAGES = {219--236},
}

@article {DongGr,
shorthand = {DG},
    AUTHOR = {Dong, C.\ and Griess, Jr., R.\@ L.\ },
     TITLE = {Automorphism groups and derivation algebras of finitely
              generated vertex operator algebras},
   JOURNAL = {Michigan Math. J.},
  FJOURNAL = {Michigan Mathematical Journal},
    VOLUME = {50},
      YEAR = {2002},
    NUMBER = {2},
     PAGES = {227--239},
}

@book {LL,
shorthand = {LL},
    AUTHOR = {Lepowsky, J.\ and Li, H.\ },
     TITLE = {Introduction to vertex operator algebras and their
              representations},
    SERIES = {Progress in Mathematics},
    VOLUME = {227},
 PUBLISHER = {Birkh\"auser Boston, Inc., Boston, MA},
      YEAR = {2004},
}

@incollection {Mason,
shorthand = {M},
    AUTHOR = {Mason, G.\ },
     TITLE = {Vertex rings and their {P}ierce bundles},
 BOOKTITLE = {Vertex algebras and geometry},
    SERIES = {Contemp. Math.},
    VOLUME = {711},
     PAGES = {45--104},
 PUBLISHER = {Amer. Math. Soc., Providence, RI},
      YEAR = {2018},
}

@book {Waterhouse,
shorthand = {W},
    AUTHOR = {Waterhouse, W.\@ C.},
     TITLE = {Introduction to affine group schemes},
    SERIES = {Graduate Texts in Mathematics},
    VOLUME = {66},
 PUBLISHER = {Springer-Verlag, New York-Berlin},
      YEAR = {1979},
}

@article {CHR,
shorthand = {CHR},
    AUTHOR = {Chase, S. U. and Harrison, D. K. and Rosenberg, A.\ },
     TITLE = {Galois theory and {G}alois cohomology of commutative rings},
   JOURNAL = {Mem. Amer. Math. Soc.},
  FJOURNAL = {Memoirs of the American Mathematical Society},
    VOLUME = {52},
      YEAR = {1965},
     PAGES = {15--33},
}

@book {Knus,
shorthand = {KO},
    AUTHOR = {Knus, M.-A.\ and Ojanguren, M.\ },
     TITLE = {Th\'eorie de la descente et alg\`ebres d'{A}zumaya},
    SERIES = {Lecture Notes in Mathematics},
    VOLUME = {389},
 PUBLISHER = {Springer-Verlag, Berlin-New York},
      YEAR = {1974},
}

@book {Kac,
shorthand = {K},
    AUTHOR = {Kac, V.},
     TITLE = {Vertex algebras for beginners},
    SERIES = {University Lecture Series},
    VOLUME = {10},
   EDITION = {Second},
 PUBLISHER = {American Mathematical Society, Providence, RI},
      YEAR = {1998},
     PAGES = {vi+201},
}

@article {Pianzola2,
shorthand = {P},
    AUTHOR = {Pianzola, A.},
     TITLE = {Vanishing of {$H^1$} for {D}edekind rings and applications to
              loop algebras},
   JOURNAL = {C. R. Math. Acad. Sci. Paris},
  FJOURNAL = {Comptes Rendus Math\'ematique. Acad\'emie des Sciences. Paris},
    VOLUME = {340},
      YEAR = {2005},
    NUMBER = {9},
     PAGES = {633--638},
}

@book {Mln,
shorthand = {Mi},
    AUTHOR = {Milne, J.\@ S.},
     TITLE = {\'{E}tale cohomology},
    SERIES = {Princeton Mathematical Series},
    VOLUME = {No. 33},
 PUBLISHER = {Princeton University Press, Princeton, NJ},
      YEAR = {1980},
}

@book {SGA3,
shorthand = {SGA3},
key = {SGA3},
     TITLE = {\textup{S\'eminaire de G\'eom\'etrie Alg\'ebrique du Bois Marie
              1962/64 (SGA 3)}, {S}ch\'emas en groupes},
    SERIES = {Lecture Notes in Mathematics},
    VOLUME = {151--153},
      NOTE = {},
 PUBLISHER = {Springer-Verlag, Berlin-New York},
      YEAR = {1970},
      EDITOR = {M. Demazure and A. Grothendieck},
}

@incollection {EM,
shorthand = {EM},
    AUTHOR = {Ein, L.\ and Musta\c{t}\u{a}, M.},
     TITLE = {Jet schemes and singularities},
 BOOKTITLE = {Algebraic geometry---{S}eattle 2005. {P}art 2},
    SERIES = {Proc. Sympos. Pure Math.},
    VOLUME = {80, Part 2},
     PAGES = {505--546},
 PUBLISHER = {Amer. Math. Soc., Providence, RI},
      YEAR = {2009},
}

@book {B,
shorthand = {Bu},
    AUTHOR = {Buium, A.\ },
     TITLE = {Differential algebra and {D}iophantine geometry},
    SERIES = {Actualit\'es Math\'ematiques. Current Mathematical Topics},
 PUBLISHER = {Hermann, Paris},
      YEAR = {1994},
}

@book {DeGa,
shorthand = {DeGa},
    AUTHOR = {Demazure, M. and Gabriel, P.},
     TITLE = {Groupes alg\'ebriques. {T}ome {I}: {G}\'eom\'etrie
              alg\'ebrique, g\'en\'eralit\'es, groupes commutatifs},
       PUBLISHER = {Masson \& Cie, \'Editeurs, Paris; North-Holland Publishing
              Co., Amsterdam},
      YEAR = {1970},
}

@incollection {Voj,
shorthand = {V},
    AUTHOR = {Vojta, P.},
     TITLE = {Jets via {H}asse-{S}chmidt derivations},
 BOOKTITLE = {Diophantine geometry},
    SERIES = {CRM Series},
    VOLUME = {4},
     PAGES = {335--361},
 PUBLISHER = {Ed. Norm., Pisa},
      YEAR = {2007},
}

@book {Se,
shorthand = {S},
    AUTHOR = {Serre, J.-P.},
     TITLE = {Galois cohomology},
 PUBLISHER = {Springer-Verlag, Berlin},
      YEAR = {1997},
}

@misc{SP,
	key = {Stacks},
	shorthand = {SP},
      title = {Stacks Project},
      note = {\texttt{https://stacks.math.columbia.edu}},
    }

@article {DLM,
shorthand = {DLM},
    AUTHOR = {Dong, C. and Li, H. and Mason, G.},
     TITLE = {Twisted representations of vertex operator algebras},
   JOURNAL = {Math. Ann.},
  FJOURNAL = {Mathematische Annalen},
    VOLUME = {310},
      YEAR = {1998},
    NUMBER = {3},
     PAGES = {571--600},
}

@article {Bor,
	shorthand = {B},
    AUTHOR = {Borcherds, R. E.},
     TITLE = {Vertex algebras, {K}ac-{M}oody algebras, and the {M}onster},
   JOURNAL = {Proc. Nat. Acad. Sci. U.S.A.},
  FJOURNAL = {Proceedings of the National Academy of Sciences of the United
              States of America},
    VOLUME = {83},
      YEAR = {1986},
    NUMBER = {10},
     PAGES = {3068--3071},
}

@article {BR,
	shorthand = {BR},
    AUTHOR = {Borcherds, R. E. and Ryba, A. J. E.},
     TITLE = {Modular {M}oonshine. {II}},
   JOURNAL = {Duke Math. J.},
  FJOURNAL = {Duke Mathematical Journal},
    VOLUME = {83},
      YEAR = {1996},
    NUMBER = {2},
     PAGES = {435--459},
}

@article {Car,
	shorthand = {C},
    AUTHOR = {Carnahan, S.},
     TITLE = {A self-dual integral form of the {M}oonshine module},
   JOURNAL = {SIGMA Symmetry Integrability Geom. Methods Appl.},
  FJOURNAL = {SIGMA. Symmetry, Integrability and Geometry. Methods and
              Applications},
    VOLUME = {15},
      YEAR = {2019},
     PAGES = {Paper No. 030, 36},
}

@incollection {Ry,
	shorthand = {R},
    AUTHOR = {Ryba, A. J. E.},
     TITLE = {Modular {M}oonshine?},
 BOOKTITLE = {Moonshine, the {M}onster, and related topics ({S}outh
              {H}adley, {MA}, 1994)},
    SERIES = {Contemp. Math.},
    VOLUME = {193},
     PAGES = {307--336},
 PUBLISHER = {Amer. Math. Soc., Providence, RI},
      YEAR = {1996},
}

@article {SS,
	shorthand = {SS},
    AUTHOR = {Schwimmer, A. and Seiberg, N.},
     TITLE = {Comments on the {$N=2,3,4$} superconformal algebras in two
              dimensions},
   JOURNAL = {Phys. Lett. B},
  FJOURNAL = {Physics Letters. B. Particle Physics, Nuclear Physics and
              Cosmology},
    VOLUME = {184},
      YEAR = {1987},
    NUMBER = {2-3},
     PAGES = {191--196},
}

@book {FLM,
	shorthand = {FLM},
    AUTHOR = {Frenkel, I. and Lepowsky, J. and Meurman, A.},
     TITLE = {Vertex operator algebras and the {M}onster},
    SERIES = {Pure and Applied Mathematics},
    VOLUME = {134},
 PUBLISHER = {Academic Press, Inc., Boston, MA},
      YEAR = {1988},
}

\small{	\noindent Mathematisches Institut, Ludwig-Maximilians-Universität München, Theresienstr.\@ 39, 80333 München, Germany; and \\ 
Department of Mathematical and Statistical Sciences, University of Alberta, Edmonton, Alberta T6G 2G1, Canada, \\
\textit{e-mail}: \textsf{rmader@ualberta.ca}

\medskip
\noindent Department of Mathematical and Statistical Sciences, University of Alberta, Edmonton, Alberta T6G 2G1, Canada, \\
\textit{e-mail}: \textsf{tjgannon@ualberta.ca}

\medskip
\noindent Department of Mathematical and Statistical Sciences, University of Alberta, Edmonton, Alberta T6G 2G1, Canada, \\
\textit{e-mail}: \textsf{a.pianzola@ualberta.ca}
}

\end{document}